\documentclass{amsart}

\usepackage{amsmath,amssymb,amsthm}
\usepackage{pinlabel}

\newtheorem{prop}{Proposition}
\newtheorem{thm}[prop]{Theorem}
\newtheorem{lem}[prop]{Lemma}

\theoremstyle{definition}

\newtheorem*{defn}{Definition}

\newtheorem{rem}[prop]{Remark}

\newtheorem*{ack}{Acknowledgements}

\def\co{\colon\thinspace}

\newcommand{\halp}{\hat{\alpha}}

\newcommand{\wB}{\widehat{B}}
\newcommand{\BB}{\mathcal{B}}

\newcommand{\C}{\mathbb{C}}

\newcommand{\D}{\mathbb{D}}
\newcommand{\rmd}{\mathrm{d}}

\newcommand{\rme}{\mathrm{e}}
\newcommand{\EE}{\mathcal{E}}

\newcommand{\F}{\mathbb{F}}
\newcommand{\Finj}{F_{\mathrm{inj}}}

\newcommand{\oG}{\overline{G}}

\newcommand{\bfh}{\mathbf{h}}

\newcommand{\rmi}{\mathrm{i}}

\newcommand{\N}{\mathbb{N}}

\newcommand{\OO}{\mathbb{O}}

\newcommand{\R}{\mathbb{R}}
\newcommand{\wR}{\widehat{\R}^{2n+1}}

\newcommand{\bfs}{\mathbf{s}}

\newcommand{\bft}{\mathbf{t}}

\newcommand{\bfv}{\mathbf{v}}

\newcommand{\utsb}{u^{\bft}_{\bfs,b}}
\newcommand{\fraku}{\mathfrak{u}}

\newcommand{\frakv}{\mathfrak{v}}
\newcommand{\frakV}{\mathfrak{V}_u}
\newcommand{\ofrakv}{\overline{\mathfrak{v}}}

\newcommand{\WW}{\mathcal{W}}
\newcommand{\wWW}{\widetilde{\mathcal{W}}}

\newcommand{\bfz}{\mathbf{z}}
\newcommand{\wZ}{\widehat{Z}}

\newcommand{\acyl}{\alpha_{\mathrm{cyl}}}
\newcommand{\loc}{\mathrm{loc}}

\DeclareMathOperator{\coker}{\mathrm{coker}}

\DeclareMathOperator{\ev}{ev}

\DeclareMathOperator{\id}{id}
\DeclareMathOperator{\im}{Im}
\DeclareMathOperator{\Int}{Int}
\DeclareMathOperator{\Index}{index}

\DeclareMathOperator{\re}{Re}

\begin{document}

\author[H.~Geiges]{Hansj\"org Geiges}
\address{Mathematisches Institut, Universit\"at zu K\"oln,
Weyertal 86--90, 50931 K\"oln, Germany}
\email{geiges@math.uni-koeln.de}
\author[K.~Zehmisch]{Kai Zehmisch}
\address{Mathematisches Institut, WWU M\"unster,
Einstein\-stra\-{\ss}e 62, 48149 M\"unster, Germany}
\email{kai.zehmisch@uni-muenster.de}

\title{Reeb dynamics detects odd balls}

\date{}

\begin{abstract}
We give a dynamical characterisation of odd-dimensional
balls within the class of all contact manifolds
whose boundary is a standard even-dimensional sphere.
The characterisation is in terms of the non-existence
of short periodic Reeb orbits.
\end{abstract}


\thanks{H.~G.\ and K.~Z.\ are partially supported by DFG grants
GE 1245/2-1 and ZE 992/1-1, respectively.}

\maketitle

%%%%%%%%%%%%%%%%%%%%%%%%%%%%%%%%%%%%%%%%%%%%%%%%%%%%%%%%%%%%%%%%%%%%%%

\section{Introduction}
\subsection{Definitions and the main result}
Let $(M,\alpha)$ be a compact, connected contact manifold
(with a fixed choice of contact form~$\alpha$) of dimension $2n+1$,
$n\in\N$, whose boundary $\partial M$ is diffeomorphic
to $S^{2n}$.

We write $\inf_0(\alpha)$ for the infimum of all positive periods of
\emph{contractible} closed orbits of the Reeb vector field~$R_{\alpha}$.
When there are no closed contractible Reeb orbits, we have
$\inf_0(\alpha)=\infty$, otherwise $\inf_0(\alpha)$ is
a minimum and in particular positive.

Our main result will be a criterion for $M$ to be diffeomorphic to
a ball in terms of $\inf_0(\alpha)$ and an embeddability
condition on $\partial M$. To formulate this condition,
we introduce the following terminology.

\begin{defn}
(a) Write $D$ for the closed unit disc in $\R^2$.
The $(2n+1)$-dimensional manifold (with boundary)
\[ Z:=\R\times D\times\C^{n-1}\]
with contact form
\[ \acyl:=\rmd b+\frac{1}{2}(x_0\,\rmd y_0-y_0\,\rmd x_0)
-\sum_{j=1}^{n-1}y_j\,\rmd x_j\]
(with the obvious denomination of cartesian coordinates)
will be referred to as the \textbf{contact cylinder}.

(b) We say that $\partial M$ admits a \textbf{contact embedding}
into the contact cylinder $Z$ if there is an embedding $\varphi$
of a collar neighbourhood of $\partial M\subset M$ into
$\Int (Z)$ with $\varphi^*\acyl=\alpha$ and with
the image of the collar under $\varphi$ contained in the
interior of $\varphi (\partial M)$.
\end{defn}

\begin{thm}
\label{thm:main}
Assume that the boundary $\partial M\cong S^{2n}$
of a contact manifold $(M,\alpha)$ as above admits
a contact embedding into the contact cylinder, and
$\inf_0(\alpha)\geq\pi$. Then $M$ is diffeomorphic to a ball.
\end{thm}

This theorem has been proved for $\dim M=3$ by Eliashberg
and Hofer~\cite{elho94}. In that paper, they also
announced the theorem for the higher-dimensional case,
but a proof has never been published. They formulated
the higher-dimensional case under the additional
homological assumption $H_2(M;\R)=0$; this condition, as we shall see,
is superfluous.

For simplicity, we shall assume throughout that $n\geq 2$,
although a large part of our argument also works for $n=1$.
Our proof shows that $(M,\alpha)$ is diffeomorphic to a ball
whenever $\partial M\cong S^{2n}$ admits a contact embedding into
the cylinder $Z_r:=\R\times D^2_r\times\C^{n-1}$ of radius $r$,
and $\inf_0(\alpha)>
\pi r^2$. Given a contact embedding into $Z=Z_1$, it may be
regarded as an embedding into a cylinder of slightly smaller
radius. Hence, even though the proof below will be based
on the assumption $\inf_0(\alpha)>\pi$, the result holds
under the weaker assumption $\inf_0(\alpha)\geq\pi$.
\subsection{Idea of the proof}
\label{subsection:idea}
The contact embedding $\varphi$ of $\partial M$ into the contact cylinder
$Z$ allows us to form a new contact manifold $\wR$ by removing the
bounded component of $\R^{2n+1}\setminus\varphi(\partial M)$
and gluing in $M$ instead. Similarly, we write $\wZ$ for the
cylinder $Z$ with $M$ glued in. We shall be studying the moduli space $\WW$ of
holomorphic discs $u=(a,f)\co \D\rightarrow \R\times\wR=:W$ in the
symplectisation $W$ of~$\wR$, where the discs are subject to certain boundary
and homological conditions. (We always write $\D$ for the closed unit disc
in $\C$ when regarded as the domain of definition of our holomorphic discs.)
It will turn out that $f(\D)$ is always contained in~$\wZ$.
We then have the following dichotomy. Either
the evaluation map
\[ \begin{array}{rccc}
\ev\co & \WW\times\D          & \longrightarrow & \wZ\\
       & \bigl( (a,f),z\bigr) & \longmapsto & f(z)
\end{array} \]
is proper and surjective, i.e.\ gives a filling, in which case
topological arguments involving the $h$-cobordism theorem
can be used to show that $M$ must be a ball.
Otherwise there will be breaking of holomorphic discs, which entails
the existence of short contractible periodic Reeb orbits
as in Hofer's paper~\cite{hofe93}.
\subsection{Remarks}
(1) The bound $\pi$ in Theorem~\ref{thm:main} is optimal.
Inside $Z$ one can form the connected sum
as described by Weinstein~\cite{wein91}, cf.~\cite[Section~6.2]{geig08},
with any contact manifold,
producing a belt sphere of radius $r_0$ smaller than, but
arbitrarily close to~$1$.
Inside this belt sphere one finds a periodic orbit of length $\pi r_0^2$.

(2) In the $3$-dimensional case, Theorem~\ref{thm:main} can be strengthened.
If $\inf_0(\alpha)\geq\pi$, then there are in fact no closed Reeb orbits
at all. (This was part of the formulation of the theorem
in~\cite{elho94}.) In this $3$-dimensional case,
the holomorphic discs project to embedded discs in $\wZ$,
where they produce a foliation by discs transverse to the Reeb direction,
see~\cite[Section~2]{elho94}. This precludes closed orbits.

(3) The existence of a foliation by discs as in (2) implies that
there cannot even be trapped Reeb orbits, i.e.\ orbits that are bounded
in forward or backward time. In \cite{grz} we show in joint work
with Nena R\"ottgen that this is a purely $3$-dimensional
phenomenon. In higher dimensions it is possible to have
a Reeb dynamics on Euclidean space, standard outside a compact set,
with trapped orbits but no periodic ones.

(4) One may consider manifolds $M$ with disconnected boundary
(and boundary components different from~$S^{2n}$).
The requirement of a contact embedding
into the contact cylinder $Z$ is made for each component of
$\partial M$ individually. By translating the images of these
components in the $\R$-direction one may then assume without loss of
generality that they are not nested. The collection $\varphi (\partial M)$
of these images is contained in a large ellipsoid $E$ inside $\Int (Z)$.
The manifold obtained from $E$ by removing the interiors of the components
of $\varphi(\partial M)$ and gluing in $M$ instead has non-trivial
fundamental group: by taking a path in $M$ joining two boundary
components, and a second path joining these two boundary points
in the exterior of $\varphi(\partial M)\subset E$, one creates an essential
loop. It follows that this manifold contains a contractible
Reeb orbit of period smaller than $\pi$. This orbit must in fact be contained
entirely in~$M$, since the Reeb flow on $Z$ is positively
transverse to any hypersurface $\{b\}\times D^{2n}$.

In other words, Theorem~\ref{thm:main} provides a means of detecting
contractible periodic orbits on non-compact manifolds or manifolds with
boundary. See \cite{bpv09,bprv,suze} for related work.
\section{Symplectisations of contactisations}
The contact cylinder $Z$ may be regarded as the contactisation
of the exact symplectic manifold $D\times\C^{n-1}\subset\C^n$. In the latter,
we have the obvious holomorphic discs $D\times\{*\}$.
In order to lift these to holomorphic discs in the symplectisation
of~$Z$, it is advantageous to proceed in two steps: first lift them
to holomorphic discs in $\C\times D\times\C^{n-1}$, and then transform
them to holomorphic discs in the symplectisation $\R\times Z$
using an explicit biholomorphism
\[ \Phi\co \R\times\R\times D\times\C^{n-1}\longrightarrow
\C\times D\times\C^{n-1}.\]
The desired boundary condition for the holomorphic discs on the
left-hand side gives us the boundary conditions for the holomorphic discs
on the right.

This allows one to transform a Cauchy--Riemann problem on the left with
respect to a `twisted' almost complex structure
(which preserves the contact hyperplanes and pairs the Reeb with
the symplectisation direction) into a Poisson problem on
a single real-valued function.

This idea is implicit in \cite[p.~1320]{elho94} and has also been
used in \cite[Proposition~5]{nied06}. Before we turn
to our specific situation, we discuss this transformation in
slightly greater generality.
\subsection{Lifting holomorphic discs}
\label{subsection:lifting}
Let $(V,J_V)$ be a Stein manifold of complex dimension~$n$.
We write $\psi$ for a plurisubharmonic potential on $V$, so
that $\omega_V:=-\rmd (\rmd\psi\circ J_V)$ is a K\"ahler form on~$V$.
In fact, what is really relevant for the following discussion is
the existence of such a potential, not the integrability
of~$J_V$, cf.~\cite[Section~3.1]{geze12}. Write $\lambda:=-\rmd\psi
\circ J_V$ for the primitive $1$-form of the symplectic form~$\omega_V$.

The contactisation of $V$ is $(\R\times V,\alpha:=\rmd b+\lambda)$,
where $b$ denotes the $\R$-coordinate. Notice that $\partial_b$ is the
Reeb vector field of the contact form~$\alpha$. A symplectisation of this
manifold is
\[ \bigl(\R\times\R\times V,\omega:=\rmd (\tau\alpha)\bigr),\]
where $\tau$ is a strictly increasing smooth
positive function on the first $\R$-factor
(whose coordinate we shall denote by~$a$). A compatible almost complex
structure $J$ on this symplectic manifold, which in addition preserves the
contact hyperplanes
\[ \ker\alpha=\{ v-\lambda(v)\partial_b\co v\in TV\} \]
on $\{ a\}\times\R\times V$, is given by
\[ J(\partial_a)=\partial_b\;\;\;\text{and}\;\;\;
J(v-\lambda(v)\partial_b)=J_V v-\lambda(J_V v)\partial_b.\]
If $J_V$ is not integrable, then $J$ may only be tamed by~$\omega$.

A straightforward calculation gives the following generalisation
of~\cite[Proposition~5]{nied06}:

\begin{prop}
\label{prop:biholo}
The map
\[ \begin{array}{rccc}
\Phi\co & (\R\times\R\times V,J) & \longrightarrow &
               (\C\times V,\rmi\oplus J_V)\\
        & (a,b,\bfz)             & \longmapsto     &
               (a-\psi(\bfz)+\rmi b,\bfz)
\end{array} \]
is a biholomorphism.\hfill\qed
\end{prop}

Given a holomorphic disc $\D\ni z\mapsto h(z)\in V$,
we want to lift this to a holomorphic disc
\[ \D\ni z\longmapsto (a(z), b(z),h(z)) \]
in the symplectisation, with boundary in
the zero level of the symplectisation, i.e.\ $a|_{\partial\D}\equiv 0$.
By Proposition~\ref{prop:biholo}, the functions $a$ and $b$
are found as follows. Let $a\co\D\rightarrow\R$ be the unique solution,
smooth up to the boundary, of the Poisson problem
\[ \left\{\begin{array}{rcll}
\Delta a & = & \Delta (\psi\circ h) & \text{on $\Int(\D)$},\\
a        & = & 0                    & \text{on $\partial\D$}.
\end{array}\right. \]
Then $a-\psi\circ h$ is harmonic, and we may choose the function $b$
(unique up to adding a constant) such that $a-\psi\circ h+\rmi b$
is holomorphic. Notice that the function $a$ is subharmonic.
\subsection{Examples}
\label{subsection:ex-disc}
(1) Our first example shows how to derive the
set-up of \cite{elho94} in this general context. We take $V=\C$
with plurisubharmonic potential $\psi(x+\rmi y)=x^2/2$. This
yields the contact form $\rmd b+x\,\rmd y$ on $\R\times\C$. Start with
the holomorphic disc $h\co\D\rightarrow \C$ given by inclusion.
The solution $a$ of the corresponding Poisson problem ---
this is equation (52) in~\cite{elho94} --- is given by
$a(x,y)=(x^2+y^2-1)/4$. For $b$ one obtains $b(x,y)=b_0-xy/2$. Notice
that $a-\psi\circ h+\rmi b$ is the holomorphic function $z\mapsto
-(z^2+1)/4+\rmi b_0$.

(2) For our second example we take $V=\C$ with plurisubharmonic potential
$\psi(z)=|z|^2/4$. This gives rise to the contact form
$\rmd b+(x\,\rmd y-y\,\rmd x)/2$ on $\R\times\C$. The solution $a$ of
the Poisson problem is unchanged, but $b$ is now simply a constant
function. The example in \cite{nied06} is obtained by
crossing this $V$ with a cotangent bundle $T^*Q$, on which one
takes the plurisubharmonic potential $\|\mathbf{p}\|^2/2$, with $\mathbf{p}$
denoting the fibre coordinate, corresponding to the canonical
Liouville $1$-form on $T^*Q$.
\subsection{The contact cylinder}
\label{subsection:cylinder}
The contact form $\acyl$ on the contact cylinder
$Z=\R\times D\times\C^{n-1}$ derives from the plurisubharmonic potential
\[ \psi (z_0;z_1,\ldots,z_{n-1}):=\frac{1}{4}|z_0|^2+
\frac{1}{2}\sum_{j=1}^{n-1}y_j^2\]
on $D\times\C^{n-1}$, where $z_j=x_j+\rmi y_j$, $j=0,1,\ldots,n-1$.

Similar to Example~\ref{subsection:ex-disc}~(2),
for any choice of parameters $b\in\R$, $\bfs,\bft\in\R^{n-1}$, we have the
holomorphic discs
\[ \begin{array}{rccc}
u^{\bft}_{\bfs,b}\co & \D & \longrightarrow & \R\times\R\times D^{2n}\\
                     & z  & \longmapsto     & \bigl(\frac{1}{4}
   (|z|^2-1),b,z,\bfs+\rmi\bft\bigr),
\end{array}\]
lifting the obvious holomorphic discs  in $D\times\C^{n-1}$.
The disc $u^{\bft}_{\bfs,b}$ has boundary on the Lagrangian
cylinder
\[ L^{\bft}:=\{0\}\times\R\times S^1\times\R^{n-1}\times\{\bft\}\]
in $\R\times Z$. These Lagrangian cylinders foliate
$\partial(\{0\}\times Z)$.
\section{The moduli space of holomorphic discs}
We now form the contact manifold $(\wR,\halp)$ as explained in
Section~\ref{subsection:idea}. Let
\[ \bigl(W:=\R\times\wR,\omega:=\rmd(\tau\halp)\bigr) \]
be its symplectisation, where $\tau\co\R\rightarrow\R^+$
is a smooth function with $\tau'>0$ and $\tau(a)=\rme^a$
for $a\geq 0$. The freedom of choosing $\tau$ on $\{ a<0\}$
is required for the asymptotic analysis cited in
Section~\ref{section:compactness}.
\subsection{The almost complex structure}
\label{subsection:complex}
Choose $b_0,r,R\in\R^+$ with $r<1$ such that $\varphi(\partial M)$
is contained in the interior of the box
\[ B:=[-b_0,b_0]\times D^2_r\times D^{2n-2}_R\subset Z,\]
where $D_{\rho}^{2k}\subset\C^k$ denotes a closed $2k$-disc
of radius~$\rho$. We write $\wB$ for the result of gluing
$M$ into this box, in other words,
\[ \wR = \wB\cup_{\partial B}\bigl((\R\times\C\times\C^{n-1})
\setminus\Int (B)\bigr).\]
We shall also have occasion to use the notation $\wZ$ for the
cylinder $Z$ with $M$ glued in, that is,
\[ \wZ=\wR\setminus \bigl( \R\times (\C\setminus\Int (D))\times
\C^{n-1}\bigr).\]

On the symplectic manifold $(W,\omega)$ we choose an almost complex
structure $J$ compatible with $\omega$ subject to the following conditions:
\begin{itemize}
\item[(J1)] On the complement of $\R\times\Int(\wB)$,
the almost complex structure
$J$ equals the one described in Section~\ref{subsection:lifting}.
\item[(J2)] On
$\R\times\Int(\wB)$,
we make a generic choice (in a sense explained in
Section~\ref{subsection:regular})
of an $\R$-invariant almost complex structure $J$ preserving $\ker\halp$
and satisfying $J(\partial_a)=R_{\halp}$.
\end{itemize}

Condition (J1) will allow us to
prove that holomorphic discs in the relevant region are standard.
Condition (J2) implies that the breaking of holomorphic discs
corresponds to cylindrical ends asymptotic to Reeb orbits.
\subsection{The moduli space}
\label{subsection:moduli}
We now consider holomorphic discs (smooth up to the
boundary) of the form
\[ u=(a,f)\co (\D,\partial\D)\longrightarrow (W=\R\times\wR,L^{\bft}), \]
i.e.\ with Lagrangian boundary condition,
where $\bft$ is allowed to vary over $\R^{n-1}$.
We shall call the value of $\bft$ corresponding to a given $u$
the `boundary level' of the holomorphic disc.

We define $\WW$ to be the moduli space of such discs~$u$, which are
supposed to satisfy the following conditions:
\begin{itemize}
\item[(M1)] The relative homology class $[u]\in H_2(W,L^{\bft})$,
with $\bft$ equal to the boundary level of~$u$,
equals that of $\utsb$ for some $b\in\R$, $\bfs\in\R^{n-1}$,
where $|b|,|\bfs|$ are large (such that $\utsb$ may be regarded
as a holomorphic disc in~$W$).
\item[(M2)] For $k=0,1,2$ we have
$u(\rmi^k)\in L^{\bft}\cap\{z_0=\rmi^k\}$.
\end{itemize}

Let $u=(a,f)$ be a holomorphic disc satisfying (M1).
By the maximum principle, $f(\D)$ is contained in~$\wZ$, see
Lemma~\ref{lem:bounds}. By the boundary lemma of E.~Hopf, applied to
a small disc in $\D$ touching a given boundary point
and mapping to the complement of $\R\times\Int(\wB)$, so that
the $z_0$-component of $u$ is defined and holomorphic on that small
disc, the boundary $u(\partial\D)$ is transverse to
\[ \{0\}\times\R\times\{\rme^{\rmi\theta}\}\times\R^{n-1}
\times\{\bft\}\subset L^{\bft}\]
for each $\rme^{\rmi\theta}\in S^1$ and, by (M1), in fact positively
transverse. Thus, condition (M2) fixes a parametrisation of~$u$.
\subsection{Properties of the holomorphic discs}
Here we collect some basic properties of the discs $u\in\WW$.

\begin{lem}
\label{lem:Maslov}
The Maslov index $\mu$ of any disc $u\in\WW$, i.e.\
the index of the bundle pair $(u^*TW,(u|_{\partial\D})^*TL^{\bft})$,
equals~$2$.
\end{lem}

\begin{proof}
We appeal to the axiomatic definition of the Maslov index
in \cite[Section~C.3]{mcsa04}. For the disc $u_0:=u^{\bft}_{{\mathbf 0},0}$
in $\R\times\R^{2n+1}$, the bundle $u_0^*T(\R\times\R^{2n+1})$ is a trivial
$\C^{n+1}$-bundle. The fibre
of the totally real subbundle $(u_0|_{\partial\D})^*TL^{\bft}$
over $\rme^{\rmi\theta}\in\partial\D$ is given by
$\R\rmi\oplus\R\rmi\rme^{\rmi\theta}\oplus\R^{n-1}$.
So the normalisation property of the Maslov index implies
$\mu (u^{\bft}_{{\mathbf 0},0})=2$.

By the homotopy invariance of the Maslov index, we have
$\mu(u^{\bft}_{\bfs,b})=2$ for all standard discs
$u^{\bft}_{\bfs,b}$ in~$W$. Finally, given any $u\in\WW$, we may choose
$u^{\bft}_{\bfs,b}$ in the same relative homology class,
so that $u-u^{\bft}_{\bfs,b}$ is a boundary. This implies
$\mu(u)=2$.
\end{proof}

\begin{lem}
\label{lem:energy}
Each disc $u\in\WW$ has symplectic energy $\int_{\D}u^*\omega$
equal to~$\pi$.
\end{lem}

\begin{proof}
Choose a standard disc $\utsb$ in the same relative class in
$H_2(W,L^{\bft})$ as~$u$. Then in particular $[\partial u]=
[\partial\utsb]$ in $H_1(L^{\bft})$. Since $L^{\bft}$ is Lagrangian,
the pull-back of the $1$-form $\halp$ to $L^{\bft}$ is closed, and hence
\[ \int_{\partial u}\halp=\int_{\partial\utsb}\halp.\]
One then computes
\[ \int_u\omega=\int_{\partial u}\halp=\int_{\partial\utsb}\halp
=\int_{\partial\utsb}\acyl=\pi.\qedhere\]
\end{proof}

\begin{rem}
\label{rem:energy}
By the same argument we see that any non-constant holomorphic disc
in $W$ with boundary on $L^{\bft}$ has symplectic energy
in $\pi\N$.
\end{rem}

\begin{lem}
\label{lem:simple}
All discs $u\in\WW$ are simple.
\end{lem}

\begin{proof}
According to \cite[Theorem~A]{lazz11}, the homology class
$[u]\in H_2(W,L^{\bft})$ of a holomorphic disc
with totally real boundary condition can be decomposed into 
positive multiples of homology classes
represented by simple discs, which are obtained from a decomposition
of~$\D$. Since the class $[u]=[\utsb]\in H_2(W,L^{\bft})$
is indecomposable by Lemma~\ref{lem:energy} and Remark~\ref{rem:energy},
the disc $u$ itself must be simple.
\end{proof}

Simplicity of the discs $u=(a,f)$ will not be quite enough for
our purposes. We shall also need simplicity of $f$ in the sense
of the following lemma, cf.~\cite[Theorem~1.14]{hwz99}.
Here $\pi$ denotes the projection of
$TM$ onto $\ker\halp$ along the Reeb vector field $R_{\halp}$.

\begin{lem}
\label{lem:f-injective}
For each $u=(a,f)\in\WW$, the set
\[ \Finj:=\bigl\{ z\in\D\co \pi\circ T_zf\neq 0,\
f^{-1}(f(z))=\{ z\}\bigr\}\]
of `$f$-injective points' is open and dense in $\D$.
\end{lem}

\begin{proof}
The combination of defining conditions for $\Finj$ is open, so we need
only show that $\Finj$ is dense in~$\D$. We begin with three
observations about the behaviour of the holomorphic discs~$u$.

First of all,
in a neighbourhood of the boundary $\partial\D\subset\D$ we can write
$f$ in components as $f=(b,\bfh)=(b,h_0,\ldots,h_{n-1})$ with
each $h_j$ holomorphic. By the comment in Section~\ref{subsection:moduli},
$h_0|_{\partial\D}$ is an immersion, hence $\pi\circ Tf|_{\partial\D}\neq 0$.
Moreover, a variant of the Carleman similarity principle
\cite[pp.~1315/6]{elho94} implies that the set $\{z\in\D\co
\pi\circ T_zf=0\}$ is finite.

Secondly,
the boundary $\partial\D$ maps under $f$ to $\R\times S^1\times\C^{n-1}$.
Near any point in $\Int (\D)$ that putatively maps to
$\R\times (\C\setminus \Int (D))\times\C^{n-1}$,
we could write $f=(b,\bfh)$ as above, and we would find
that $h_0$ violates the maximum principle. We conclude in particular that
there are no mixed intersections of the holomorphic disc~$u$,
i.e.\ pairs of an interior and a boundary point with the same
image.

Thirdly, from the work in \cite{zehm13} it follows that the immersion
$u|_{\partial\D}= (0,f|_{\partial\D})$ has at most finitely
many double points. Otherwise the respective preimages would
accumulate in two separate points --- for in a common limit point
the differential $Tu$ would be singular --- and
\cite[Lemma~4.2]{zehm13} would imply that the differentials $Tu$ in the
two limit points are collinear over~$\R$. Furthermore, by
Lemma~\ref{lem:bounds}~(i) below, the collinearity factor
would have to be positive. Then \cite[Lemma~4.3]{zehm13} would imply that
$u$ is not simple, contradicting the preceding lemma.

From these last two observations we infer that $\Finj$ contains
$\partial\D$ with the exception of at most finitely many points, and in
particular is non-empty.

Now we prove that $\Finj$ is dense, arguing by contradiction.
If $\Finj$ were not dense, the set $\Int(\D)\setminus\Finj$
would have non-empty interior. By the preceding observations
we can find an open subset $U\subset\Int(\D)$ such that
for each $z\in U$ the set $f^{-1}(f(z))\subset\Int(\D)$ contains more than
just the point~$z$, and such that $\pi\circ T_wf\neq 0$ in all points
$w\in f^{-1}(f(U))$. The latter implies that the points
in $f^{-1}(f(z))$ are isolated, and hence finite in number.

What follows is an explication of an argument in \cite[p.~459]{hwz99}.
Fix a point $z_0\in U$ and write $f^{-1}(f(z_0))=\{z_0,z_1,\ldots, z_N\}$.
Choose pairwise disjoint (and disjoint from~$U$) open neighbourhoods
$U_k\subset\Int(\D)$ of $z_k$, $k=1,\ldots, N$,
such that $f|_{U_k}$ is an embedding.
By a compactness argument,
$U$ can be chosen so small that
\[ f(U)\subset \bigcup_{k=1}^N f(U_k),\]
and such that $f|_{U}$ is likewise an embedding. Choose relatively
compact neighbourhoods $U_k'\subset U_k$ of $z_k$, $k=1,\ldots, N$.
By shrinking $U$ to a smaller neighbourhood of~$z_0$, we can ensure that
\[ f(U)\subset \bigcup_{k=1}^N f(U_k').\]
Set $A_k:=(f|_U)^{-1}(f(U_k'))\subset U$. If $A_1$ has non-empty
interior, we can shrink $U$ such that $f(U)\subset f(U_1')$
(but $U$ need no longer be a neighbourhood of~$z_0$).
The argument then concludes as in \cite[pp.~459/60]{hwz99},
leading to a contradiction to $u$ being simple.
If $A_1$ has empty interior, so that $U\setminus A_1$ is dense in $U$,
we find that
\[ f(U)\subset \bigcup_{k=2}^N f(\overline{U'_k})\subset
 \bigcup_{k=2}^N f(U_j).\]
The argument concludes inductively.
\end{proof}
\subsection{Bounds on the holomorphic discs}
In the next lemma we collect some restrictions on the image $u(\D)$
of the holomorphic discs $u\in\WW$.

\begin{lem}
\label{lem:bounds}
For $u=(a,f)\in\WW$ we have:
\begin{itemize}
\item[(i)] $a<0$ on $\Int (\D)$.
\item[(ii)] $f(\Int(\D))$ is contained in the interior of $\wZ$, i.e.\
\[ f(\Int(\D))\cap\bigl(\R\times(\C\setminus\Int(D))
\times\C^{n-1}\bigr)=\emptyset.\]
\end{itemize}
\end{lem}

\begin{proof}
(i) The holomorphicity of $u=(a,f)$ (with respect to an almost complex
structure preserving $\ker\halp$ and satisfying $J(\partial_a)=R_{\halp}$)
implies $f^*\halp=-\rmd a\circ\rmi$, so $a$ is subharmonic.
We have $a|_{\partial\D}\equiv 0$, but $a$ cannot be identically zero
on all of $\D$, for otherwise we would have $f^*\alpha\equiv 0$
and $f^*\rmd\alpha\equiv 0$, which would imply that $u$ has
zero symplectic energy density and hence is constant, contradicting (M1).
The strong maximum principle for $a$ then implies the claim.

(ii) Near the points of $\D$ mapping under $f$ to
$\R\times(\C\setminus\Int(D))\times\C^{n-1}$
we can write this map in components as $f=(b,\bfh)$.
If $f(\Int(\D))$ were not contained in $\Int (Z)$, we would find that
the map $h_0$ is defined and locally constant on a non-empty open and closed
subset of~$\D$, and hence on all of~$\D$, contradicting the
homological assumption (M1).
\end{proof}

Since a generic choice of the almost complex structure $J$
is only allowed on $\R\times\Int(\wB)$, this can be used
to guarantee regularity in the sense of \cite[Definition~3.1.4]{mcsa04}
only for those holomorphic discs that pass through
this `perturbation domain', see~\cite[Remark~3.2.3]{mcsa04}.
We therefore want to show that all
other discs belong to the standard family $\utsb$, where transversality
is obvious. This will be used below to show that $\WW$ is actually
a manifold.

\begin{lem}
\label{lem:standard1}
Let $u=(a,f)\in\WW$.
If $f(\D)\subset\wR\setminus\Int(\wB)$,
then $u=\utsb$ for some $\bfs\in\R^{n-1}$,
$b\in\R$, and $\bft$ equal to the boundary level of~$u$.
\end{lem}

\begin{proof}
Since $f$ maps to the complement of $\Int(\wB)$, we can write it
\emph{globally} as $f=(b,\bfh)$,
with every component $h_j$  of $\bfh$ a holomorphic map $\D\rightarrow\C$.
The boundary condition for $u$ means that for $j=1,\ldots ,n-1$ we have
$\im h_j=t_j$ on $\partial\D$. The minimum and maximum principle for harmonic
functions implies that $\im h_j=t_j$ on all of~$\D$. Hence, by the
open mapping theorem, $\re h_j=:s_j$ is likewise constant on $\D$
for $j=1,\ldots,n-1$.

The component $h_0$ is a holomorphic disc in $\C$ with
$h_0|_{\partial\D}$ an orientation-preserving diffeomorphism
of $\partial\D$, cf.\ the comment after condition (M2).
The argument principle implies that
$h_0$ is an orientation-preserving automorphism
of~$\D$, and then (M2) forces $h_0=\id_{\D}$.

By Proposition~\ref{prop:biholo}, the function
\[ z\longmapsto a(z)-\frac{1}{4}|h_0(z)|^2-\frac{1}{2}\sum_{j=1}^{n-1}
(\im h_j)^2
=a(z)-\frac{1}{4}|z|^2-\frac{1}{2}\sum_{j=1}^{n-1}
(\im h_j)^2\]
on $\D$ is harmonic, taking the constant value $-1/4-|\bft|^2/2$
on $\partial\D$, hence it is constant on $\D$. This means
that the imaginary part $b$ that makes this into a holomorphic
function must also be constant. Solving for $a(z)$ we get
\[ a(z)=\frac{1}{4}(|z|^2-1)\;\;\text{on $\D$},\]
i.e.\ $u=\utsb$.
\end{proof}

The next lemma will allow us to control the degree of the evaluation
map~$\ev$. It says hat non-standard disc can never reach $b$-levels
with $|b|>b_0$. This is also relevant for compactness.

\begin{lem}
\label{lem:standard2}
Let $u=(a,f)\in\WW$. On the closed set $A:=
f^{-1}(\wR\setminus\Int(\wB))\subset\D$,
which includes the whole boundary $\partial\D$ in its
interior, we write $f=(b,\bfh)$. If the function $b$ takes values
outside $[-b_0,b_0]$,
then $f$ maps to a $b$-level set $\{b_1\}\times D\times\C^{n-1}$
with $|b_1|>b_0$ and hence, by the preceding lemma, the holomorphic curve
$u$ equals $u_{\bfs,b_1}^{\bft}$ for some $\bfs$ and~$\bft$.
\end{lem}

\begin{proof}
Choose $z_*\in A$ with $b_*:=b(z_*)$ of maximal absolute value.
Notice that $z_*$ is an interior point of~$A$.
By Proposition~\ref{prop:biholo}, the function
\[ g:= a-\frac{1}{4}|h_0|^2-\frac{1}{2}\sum_{j=1}^{n-1}(\im h_j)^2+\rmi b\]
is holomorphic on~$\Int(A)$. We should now like to argue with the maximum
principle that the imaginary part $b$ of $g$ has to be constant equal to
$b_*$ on an open and closed
subset of~$\D$. If $z_*\in\Int(\D)$, this inference is indeed
conclusive, just as in part (ii) of Lemma~\ref{lem:bounds}.
If $z_*\in\partial\D$, we reason as follows.

The real part of the holomorphic function $g$
takes the constant value $a_{\partial}:=-1/4-|\bft|^2/2$ on
$\partial\D\subset\Int(A)$.
It follows that the function can be extended by Schwarz reflection
to the complementary set $\overline{A}$
of $A$ in $\hat{\C}\setminus\D$, with $\hat{\C}$
denoting the Riemann sphere. Indeed, the holomorphic function
$\rmi (g-a_{\partial})$ takes real values on $\partial\D$,
so the Schwarz reflection principle applies to this function, and
we simply transform the extension via the map $w\mapsto -\rmi w+a_{\partial}$
to a holomorphic extension of~$g$. Now $z_*$ is an interior point
of $A\cup\overline{A}$, and we conclude as before with the
maximum principle.
\end{proof}

Finally, we establish a $C^0$-bound in the
$\C^{n-1}$-direction on non-standard discs.

\begin{lem}
\label{lem:standard3}
Let $u=(a,f)\in\WW$. If $f(\D)$ intersects
\[ \R\times\C\times \bigl(\C^{n-1}\setminus D^{2n-2}_{R+\sqrt{2}}\bigr),\]
then $u$ equals one of the standard discs~$\utsb$.
\end{lem}

\begin{proof}
Consider the open subset
\[ G:= f^{-1}\bigl(\R\times\C\times\bigl(\C^{n-1}\setminus D^{2n-2}_R
\bigr)\bigr)\subset\D,\]
which will be non-empty under the assumption on $f$ in the lemma.
On the closure $\oG$ of $G$ we write $f=(b,\bfh)$ as before and consider
the subharmonic function $h:=|h_1|^2+\cdots+|h_{n-1}|^2$.

Write $\partial G$ for the topological boundary of $G$ in~$\D$.
We have $h|_{\partial G}\equiv R^2$, so the maximum of $h$
on $\oG$ must be attained at a point in $G\cap\partial\D$ (in particular,
this intersection must be non-empty).

If $\oG=\D$, we are done by Lemma~\ref{lem:standard1}.
Otherwise, we perform Schwarz reflection on the holomorphic function
$h_j-\rmi t_j$, which is possible since $\im h_j\equiv t_j$ on
$\partial\D$. To the extended function
we add $\rmi t_j$ again to obtain the extension of~$h_j$.
Geometrically, this corresponds to a reflection of $h_j(G)$
in the line $\{z_j=\rmi t_j\}\subset\C$.

Write $S$ for the compact subset of the Riemann sphere given as the union
of $\oG$ and its reflected copy, and continue to write
$h$ for the extension of the plurisubharmonic function to~$S$.
Beware that $h$ may take larger values on $S$ than on $\oG$.

Choose a point $s_0\in G\cap\partial D$ where $h|_{\oG}$
attains its maximum $(R+\delta)^2$. Now consider an
open $\delta$-ball $B_{\delta}$ about the point $\bfh(s_0)\in\C^{n-1}$.
Then $\bfh(\partial G)$ is contained in the complement
of~$B_{\delta}$, and since the extension of $\bfh$ to $S$
was obtained by Schwarz reflection along $\partial\D\ni s_0$,
the full boundary $\bfh(\partial S)$ after reflection
will likewise be contained in the complement of~$B_{\delta}$.

This allows us to apply the monotonicity lemma~\cite[Theorem~1.3]{humm97},
which tells us that the area of $\bfh(S)\cap\overline{B}_{\delta}$
is bounded from below by $\pi\delta^2$. (In \cite{humm97}
the estimate is given in the form $\text{const.}\cdot\delta^2$;
in the present Euclidean setting the constant $\pi$ comes
from the classical isoperimetric inequality.)
So the area of $\bfh(\oG)\cap\overline{B}_{\delta}$ is bounded
from below by $\pi\delta^2/2$, and from above by the energy $\pi$ of~$u$.
This implies $\delta\leq\sqrt{2}$.

To sum up: Any holomorphic disc $u$ whose $\bfh$-component
stays outside $D^{2n-2}_R$ is standard; for all other discs
the $\bfh$-component stays inside $D^{2n-2}_{R+\sqrt{2}}$.
\end{proof}
\section{Compactness}
\label{section:compactness}
In this section we establish, under the assumption $\inf_0(\alpha)>\pi$,
compactness of the \emph{truncated} moduli space
\[ \WW':=\bigl\{ u=(a,f)\in\WW\co f(\D)\subset
[-b_0,b_0]\times D\times D^{2n-2}_{R+\sqrt{2}}\bigr\},\]
i.e.\ the space obtained from $\WW$ by cutting off ends
containing standard discs only.
\subsection{Variable boundary condition}
\label{subsection:variable}
The holomorphic discs $u\in\WW$ have boundary on
the Lagrangian cylinder $L^{\bft}$, which varies with the
parameter $\bft\in\R^{n-1}$.
It is possible to fix the boundary condition, at the cost of allowing the
almost complex structure to vary. This is done with the
help of a flow that enables us to identify different
copies of $L^{\bft}$. That flow will also provide
explicit charts when we discuss transversality.

Start with a constant vector field $\bfv$ on the space $\im\C^{n-1}$
of $\bft$-coordinates, and regard this as a vector field
on $\R\times\R\times\C\times\C^{n-1}$. Cut this off with a bump function
supported near
\[ \{0\}\times [-b_0,b_0]\times S^1\times\C^{n-1}\]
and identically $1$ in a smaller neighbourhood of that set.
Then write $\psi_t^{\bfv}$ for the flow of this vector field.

For a sequence $u_{\nu}$ of holomorphic discs of level
$\bft_{\nu}\rightarrow \bft_0$, we can then use the maps
$\psi_1^{\bft_{\nu}-\bft_0}$ to pull back the $u_{\nu}$ to
$J_{\nu}$-holomorphic discs of level $\bft_0$, where
$J_{\nu}:=(\psi_1^{\bft_{\nu}-\bft_0})^*J$ is $C^{\infty}$-convergent to $J$
and coincides with $J$ outside the neighbourhood described in the
preceding paragraph.
\subsection{Proof of compactness}
Now we apply this construction to the truncated moduli space $\WW'$.
Consider a sequence $(u_{\nu})$ of
holomorphic discs $u_{\nu}=(a_{\nu},f_{\nu})\in\WW'$.
Then, in particular the levels $\bft_{\nu}$ will be contained in
the compact set $D^{n-1}_{R+\sqrt{2}}$. Hence,
after passing to a subsequence, we may assume that $\bft_{\nu}\rightarrow
\bft_0$ for some $\bft_0\in D^{n-1}_{R+\sqrt{2}}$.
With the construction from the
preceding section we may take the $u_{\nu}$ to be
$J_{\nu}$-holomorphic discs of fixed boundary level~$\bft_0$.
The almost complex structures $J_{\nu}$ equal $J$ outside
a neighbourhood of $\{0\}\times [-b_0,b_0]\times S^1\times\C^{n-1}$
and converge to~$J$ in the $C^{\infty}$-topology.
By Lemma~\ref{lem:energy}, all discs $u_{\nu}$ have symplectic energy
equal to~$\pi$.

We claim that there  is a uniform bound on  $\max_{\D}|\nabla u_{\nu}|$.
Here $|\,.\,|$ denotes the norm
corresponding to an $\R$-invariant metric on $W$ of the form
$\rmd a^2+g_{\wR}$, with $g_{\wR}$ any Riemannian metric on~$\wR$.
The mean value theorem then gives a uniform $C^0$-bound
on $(a_{\nu})$, and compactness follows as in \cite{geze13} with
\cite[Theorem~B.4.2]{mcsa04}.

Bubbling off analysis as in \cite[Section~6]{geze13} shows that,
\emph{a priori}, the following phenomena might occur:
\begin{itemize}
\item[-] bubbling of spheres
\item[-] bubbling of finite energy planes
\item[-] breaking
\item[-] bubbling of discs (this can only happen at boundary points).
\end{itemize}
The first is impossible in an exact symplectic manifold. The second
and third phenomenon are precluded by the assumption
$\inf_0(\alpha)>\pi$ and the energy estimate from Lemma~\ref{lem:energy},
cf.~\cite[p.~584]{geze13}, since a finite energy plane in
a symplectisation is asymptotic to a \emph{contractible} Reeb orbit.
Notice that this rules out any kind of bubbling
at interior points.

This leaves the bubbling of discs at boundary points.
By Remark~\ref{rem:energy}, there could be at best a single
bubble disc at the boundary, taking away the full energy~$\pi$,
cf.~\cite[Theorem~4.6.1]{mcsa04}.
But the $C^{\infty}_{\loc}$ convergence on the complement of the
bubble point, together with condition~(M2), is
incompatible with a ghost disc.
\section{Transversality}
\label{section:transversality}
The purpose of this section is to show that the truncated moduli
space $\WW'$ is a smooth, oriented manifold with boundary.
As usual, this is achieved by proving transversality results in
the setting of $W^{1,p}$-maps for some $p>2$. Smoothness
of the holomorphic discs is then implied by elliptic regularity.

Let $\BB$ denote the space of $W^{1,p}$-maps
\[ u\co (\D,\partial\D)\longrightarrow (W,\{0\}\times\wR),\]
where $u(\partial\D)$ is supposed to be contained in $L^{\bft}$
for some $\bft\in\R^{n-1}$, and $u$ is
required to satisfy the homological condition (M1) from
Section~\ref{subsection:moduli}. Write $\BB^{\bft}\subset\BB$
for the subspace of discs corresponding to a fixed
boundary level~$\bft$.

The space $\BB^{\bft}$ is a (separable) Banach manifold 
modelled on the Banach space of $W^{1,p}$-sections
of $u^*(TW,TL^{\bft})$ (i.e.\ vector fields along $u$ that are
tangent to $L^{\bft}$ along the boundary); charts are obtained from
such vector fields along $u$ by choosing
a metric for which the submanifold $L^{\bft}$ is totally
geodesic and then applying the exponential map,
see~\cite{elia67}.
The construction from Section~\ref{subsection:variable}
shows that the map sending a disc $u\in\BB$ to its level $\bft$
gives $\BB$ the structure of a locally trivial fibration
over $\R^{n-1}$ with fibre $\BB^{\bft}$.
Tangent vectors at $u\in\BB$ can be written uniquely
as $\fraku+\frakv|_u$,
where $\fraku\in T_u\BB^{\bft}$, and $\frakv$
is a vector field as in Section~\ref{subsection:variable}
coming from a constant vector field $\ofrakv$ on
$\im\C^{n-1}$.
\subsection{The linearised Cauchy--Riemann operator}
Over $\BB$ we have a Banach space bundle $\EE$ whose
fibre over the point $u\in\BB$ is the space
$L^p(u^*TW)$ of $L^p$-vector fields along~$u$; see
for instance \cite[Proposition~6.13]{abklr94} for the construction
of the bundle structure. This bundle inherits the local product structure
from~$\BB$.

Fix an almost complex structure $J$ on $W$ satisfying
the conditions (J1) and (J2).
The Cauchy--Riemann operator $u\mapsto u_x+J(u)u_y$
defines a section of~$\EE$. In order to discuss transversality,
we need to compute the vertical differential $D_u$ of this section
at $u\in\BB$. To this end, consider a path of holomorphic curves
\[ u^s:=\psi_1^{s\frakv}\circ\exp_u(s\fraku)\]
for $s$ in some small interval around~$0$, where $\psi$ denotes the
flow as in Section~\ref{subsection:variable}. This path is
tangent to $\fraku+\frakv|_u$ in $s=0$. Let $\nabla$ be
a torsion-free connection on $TW$. Write
\[ \nabla_s=\bigl(\nabla_{\partial u^s/\partial s}\bigr)|_{s=0},\;\;\;
\nabla_x=\bigl(\nabla_{\partial u^s/\partial x}\bigr)|_{s=0},\]
and likewise $\nabla_y$. Since the torsion of $\nabla$ vanishes,
we have
\[ \nabla_s\frac{\partial u^s}{\partial x}=
\nabla_x\frac{\partial u^s}{\partial s}=\nabla_x(\fraku+\frakv),\]
and similarly for $\partial u^s/\partial y$.
Hence
\begin{eqnarray*}
D_u(\fraku+\frakv|_u) & = & \nabla_s(u^s_x+J(u^s)u^s_y)\\
  & = & \nabla_x(\fraku+\frakv)+J(u)\nabla_y(\fraku+\frakv)
        +\bigl(\nabla_{\fraku+\frakv}J\bigr)(u)\, u_y\\
  & = & D_u^{\bft}\fraku+K_u\ofrakv,
\end{eqnarray*}
where
\begin{eqnarray*}
D_u^{\bft}\fraku & := & \nabla_x\fraku+J(u)\nabla_y\fraku+
                        \bigl(\nabla_{\fraku}J\bigr)(u)u_y,\\
K_u\ofrakv       & := & \nabla_x\frakv+J(u)\nabla_y\frakv+
                        \bigl(\nabla_{\frakv}J\bigr)(u)u_y.
\end{eqnarray*}
The operator $\ofrakv\mapsto K_u\ofrakv$ is linear of order $0$
in $\ofrakv$, and hence a compact operator. The restriction
of $D_u$ to the subspace $T_u\BB^{\bft}$ equals $D_u^{\bft}$,
which is a Fredholm operator of index
\[ \Index(D_u^{\bft})=\mu+n+1=n+3\]
by the index formula \cite[Theorem~C.1.10]{mcsa04}
and Lemma~\ref{lem:Maslov}. The subspace $\frakV\subset T_u\BB$
made up of vectors of the form $\frakv|_{u}$ is
$(n-1)$-dimensional, and it is contained in the kernel of $D_u^{\bft}$.
Hence, by the invariance under compact
perturbations of both the Fredholm property and the
index, see \cite[Theorem~A.1.4]{mcsa04}, we have ---
writing $\OO$ for the zero operator ---
\[ \Index(D_u)=\Index(D_u^{\bft}+\OO_{\frakV})=
\Index(D_u^{\bft})+n-1=2n+2.\]
\subsection{Regular almost complex structures}
\label{subsection:regular}
Given an almost complex structure $J$ on $W$ subject to the
constraints (J1) and (J2), write $\wWW$ for the space
of \emph{holomorphic} discs $u\in\BB$, i.e.\
those $u$ with $u_x+J(u)u_y=0$. In other words, these
are holomorphic discs satisfying condition (M1).

The almost complex structure $J$ is called \emph{regular} if
two conditions are satisfied:
\begin{itemize}
\item[(i)] $D_u$ is onto for all $u\in\wWW$.
\item[(ii)] The evaluation map
\[ \begin{array}{ccc}
\wWW       & \longrightarrow & L^{\bft}\times L^{\bft}\times L^{\bft}\\
u=(a,f)    & \longmapsto     & (f(1),f(\rmi),f(-1))
\end{array}\]
is transverse to $(\R\times L^{\bft}_1)\times(\R\times L^{\bft}_{\rmi})
\times(\R\times L^{\bft}_{-1})$, where $L^{\bft}_{e^{\rmi\theta}}:=
L^{\bft}\cap\{z_0=e^{\rmi\theta}\}$.
\end{itemize}
If the first condition is satisfied, $\wWW$ will
be a manifold of the expected dimension $2n+2$; if in addition (ii)
holds, then $\WW$ will be a manifold of dimension $2n-1$.

The proof that the set of regular $J$ is non-empty, in fact of second
Baire category, follows the standard line of reasoning as
in the proof of Theorems~3.1.5 and 3.4.1 of~\cite{mcsa04}.
Selecting such a regular $J$ is the generic choice we make
in~(J2). For the standard discs $u^{\bft}_{\bfs,b}$,
transversality is obvious. By Lemma \ref{lem:standard1},
all discs that are not standard
pass through the region where $J$ may be chosen generically,
which is sufficient to achieve transversality
by \cite[Remark~3.2.3]{mcsa04}. In contrast with the
set-up in~\cite{mcsa04}, we are
only allowed to perturb $J$ along $\xi$, keeping it compatible
with~$\rmd\halp$. But this is exactly the situation dealt with
by Bourgeois in the appendix of~\cite{bour06}. The proof given there
carries over to our situation; the essential ingredient of
Bourgeois's argument is that the set of $f$-injective points
is open and dense, which is precisely our Lemma~\ref{lem:f-injective}.
\subsection{Orientation}
In order to speak of the degree of the evaluation map
$\ev$ on $\WW\times\D$, we need to put an orientation on the
moduli space~$\WW$. Given the relation between $\WW$ and
$\wWW$ described in the preceding section, it suffices to
orient~$\wWW$, and that in turn amounts to showing
that the determinant line bundle $\det D$ over $\wWW$ is oriented,
since $\ker D_u=T_u\wWW$.

Recall that the determinant line $\det F$ is defined for any
Fredholm operator $F$ as $\det F=\det\ker F\otimes
(\det\coker F)^*$. Since $D_u$ is surjective for all $u\in\wWW$,
the determinant line bundle is simply $\det\ker D=\bigwedge^{2n+2}\ker D$.
In the arguments that follow, however, we use deformations
through not necessarily surjective Fredholm operators,
so we need to work with determinant lines, in general.

As we have seen, the operator $D_u$ splits (by slight abuse of notation)
as $D_u=D_u^{\bft}+K_u$. The linear interpolation of $D_u$
to $D_u^{\bft}+\OO_{\frakV}$ is via Fredholm operators, since $K_u$ is compact.
It follows that $\det(D_u)=\det(D_u^{\bft}+\OO_{\frakV})$,
see~\cite[p.~680]{fooo09}, whence
\[ \det(D_u)=\det D_u^{\bft}\otimes\det\frakV.\]
The second factor inherits a natural orientation from the
orientation of $\R^{n-1}$. The first factor is naturally
oriented by the construction in \cite[Section~8.1]{fooo09}.
Our situation is a particularly simple one, since $TL^{\bft}$
is a trivial bundle. This implies that any bundle pair
$(u^*TW,(u|_{\partial\D})^*TL^{\bft})$ comes with
a natural trivialisation of the boundary bundle, and this suffices
for the construction of a natural orientation of the determinant line
bundle.
\section{Proof of Theorem~\ref{thm:main}}
By Sections \ref{section:compactness} and~\ref{section:transversality}
(notably Lemma~\ref{lem:standard2}),
the assumption $\inf_0(\alpha)>\pi$
of Theorem~\ref{thm:main} implies that the evaluation map
\[ \begin{array}{rccc}
\ev\co & \WW\times\D          & \longrightarrow & \wZ\\
       & \bigl( (a,f),z\bigr) & \longmapsto     & f(z)
\end{array} \]
is a proper map of degree~$1$. By Lemmata \ref{lem:standard1}
and~\ref{lem:standard3}, we may pretend that
$\WW\times\D$ and $\wZ$ are --- after smoothing corners ---
compact, oriented manifolds with boundary, without changing the homotopy type
of these spaces, and that $\ev$ is
a smooth degree~$1$ map between these manifolds.

Homotopical and homological arguments similar to the ones that follow
were used by Eliashberg--Floer--McDuff, see~\cite{mcdu91}.

\begin{prop}
\label{prop:pi1}
The manifold $\wZ$ is simply connected.
\end{prop}

\begin{proof}
Given a loop in $\wZ$, we homotope it to an embedded circle $C$
inside $\Int (\wZ)$ that intersects the complement of~$\wB$,
in other words, such that it passes through the region where
all holomorphic discs (more precisely, their $f$-components)
are standard. We can make the evaluation map
\[ \begin{array}{ccc}
\WW\times\D          & \longrightarrow & \wZ\\
\bigl( (a,f),z\bigr) & \longmapsto     & f(z)
\end{array} \]
transverse to $C$ by a perturbation compactly supported
in $\Int (\wB)$. The preimage of $C$ under this perturbed map
will then be a single circle $C'\subset\WW\times\D$
mapping with degree 1 onto~$C$. The homotopy of $C'$ to
a loop in $\WW\times\{1\}$ induces a homotopy of $C$
to a loop in the cell $\R\times\{1\}\times\C^{n-1}$.
\end{proof}

\begin{lem}
Let $\phi\co (P,\partial P)\rightarrow (Q,\partial Q)$ be a degree~$1$
map between compact, oriented $m$-dimensional manifolds with boundary.
Then the induced homomorphism $\phi_*\co H_k(P;\F)\rightarrow
H_k(Q;\F)$ in singular homology with coefficients
in a field $\F$ is surjective in each degree $k\in\N_0$.
\end{lem}

\begin{proof}
Over a field, the Kronecker pairing between homology and cohomology is
non-degenerate, so equivalently we need to show injectivity
of the induced homomorphism $\phi^*$ in cohomology.

Given a non-zero class $\beta\in H^k(Q)$, Poincar\'e duality allows us to
find a class $\gamma\in H^{m-k}(Q,\partial Q)$ such that $\beta\cup\gamma$
is the orientation generator of $H^m(Q,\partial Q)$. Since $\phi$ is of
degree~$1$, we have
\[ 0\neq\phi^*(\beta\cup\gamma)=\phi^*\beta\cup\phi^*\gamma,\]
which forces $\phi^*$ to be injective on $H^k(Q)$.
\end{proof}

\begin{prop}
\label{prop:H*}
The manifold $\wZ$ has the integral homology of a point.
\end{prop}

\begin{proof}
With the preceding lemma this follows with an argument completely analogous
to the proof of Proposition~\ref{prop:pi1}.
\end{proof}

\begin{proof}[Proof of Theorem~\ref{thm:main}]
Since $2n\neq 3$, the smooth Schoenflies theorem tells us that
the subset of $\wZ$ bounded by $\varphi(\partial M)$ and a standard
ellipsoid surrounding $\varphi(\partial M)$ is diffeomorphic to
a collar of $\partial M$. Hence $M$ is a strong deformation retract
of~$\wZ$. So by Propositions \ref{prop:pi1} and~\ref{prop:H*},
the manifold $M$ is a simply connected homology ball with boundary
diffeomorphic to~$S^{2n}$. It follows that $M$ is diffeomorphic to
a ball: for $n\geq 3$ we appeal to Proposition~A on page 108
of Milnor's lectures~\cite{miln65}; for $n=2$, to Proposition~C on
page~110.
\end{proof}
\begin{ack}
We thank Peter Albers for useful conversations about compactness
questions, and Chris Wendl for drawing our attention to Fr\'ed\'eric
Bourgeois's work on transversality in the setting of symplectisations.
A part of the work on this paper was done during the workshop on
Legendrian submanifolds, holomorphic curves and generating families at
the Acad\'emie Royale de Belgique, August 2013, organised by
Fr\'ed\'eric Bourgeois.
\end{ack}


\begin{thebibliography}{10}
%
\bibitem{abklr94}
{\sc B. Aebischer, M. Borer, M. K\"alin, Ch. Leuenberger
and H. M. Reimann},
\textit{Symplectic Geometry},
Progr. Math. \textbf{124},
Birkh\"auser, Basel (1994).
%
\bibitem{bpv09}
{\sc J. B. van den Berg, F. Pasquotto and R. C. Vandervorst},
Closed characteristics on non-compact hypersurfaces in $\R^{2n}$,
\textit{Math. Ann.}
\textbf{343} (2009), 247--284.
%
\bibitem{bprv}
{\sc J. B. van den Berg, F. Pasquotto, T. O. Rot and R. C. Vandervorst},
Closed characteristics on non-compact mechanical contact manifolds,
preprint (2013),
{\tt arXiv:~1303.6461}.
%
\bibitem{bour06}
{\sc F. Bourgeois},
Contact homology and homotopy groups of the space of contact structures,
\textit{Math. Res. Lett.}
\textbf{13} (2006), 71--85.
%
\bibitem{elho94}
{\sc Ya. Eliashberg and H. Hofer},
A Hamiltonian characterization of the three-ball,
\textit{Differential Integral Equations}
\textbf{7} (1994), 1303--1324.
%
\bibitem{elia67}
{\sc H. I. El\'{\i}asson},
Geometry of manifolds of maps,
\textit{J. Differential Geometry}
\textbf{1} (1967), 169--194.
%
\bibitem{fooo09}
{\sc K. Fukaya, Y.-G. Oh, H. Ohta and K. Ono},
\textit{Lagrangian Intersection Floer Theory: Anomaly and Obstruction,
Part II},
AMS/IP Stud. Adv. Math.
\textbf{46.2},
American Mathematical Society, Providence, RI (2009).
%
\bibitem{geig08}
{\sc H. Geiges},
\textit{An Introduction to Contact Topology},
Cambridge Stud. Adv. Math.
\textbf{109},
Cambridge University Press (2008).
%
\bibitem{grz}
{\sc H. Geiges, N. R\"ottgen and K. Zehmisch},
Trapped Reeb orbits do not imply periodic ones,
\textit{Invent. Math.},
to appear.
%
\bibitem{geze12}
{\sc H. Geiges and K. Zehmisch},
Symplectic cobordisms and the strong Weinstein conjecture,
\textit{Math. Proc. Cambridge Philos. Soc.}
\textbf{153} (2012), 261--279.
%
\bibitem{geze13}
{\sc H. Geiges and K. Zehmisch},
How to recognize a $4$-ball when you see one,
\textit{M\"unster J. Math.}
{\bf 6} (2013), 525--554; erratum: pp.~555--556.
%
\bibitem{hofe93}
{\sc H. Hofer},
Pseudoholomorphic curves in symplectizations with applications to the
Weinstein conjecture in dimension three,
\textit{Invent. Math.}
\textbf{114} (1993), 515--563.
%
\bibitem{hwz99}
{\sc H. Hofer, K. Wysocki and E. Zehnder},
Properties of pseudoholomorphic curves in symplectizations III:
Fredholm theory, in
\textit{Topics in Nonlinear Analysis},
Progr. Nonlinear Differential Equations Appl. {\bf 35},
Birkh\"auser, Basel (1999), 381--475.
%
\bibitem{humm97}
{\sc C. Hummel},
\textit{Gromov's Compactness Theorem for Pseudo-holomorphic Curves},
Progr. Math. {\bf 151},
Birkh\"auser, Basel (1997).
%
\bibitem{lazz11}
{\sc L. Lazzarini},
Relative frames on $J$-holomorphic curves,
\textit{J. Fixed Point Theory Appl.}
\textbf{9} (2011), 213--256.
%
\bibitem{mcdu91}
{\sc D. McDuff},
Symplectic manifolds with contact type boundaries,
\textit{Invent. Math.}
\textbf{103} (1991), 651--671.
%
\bibitem{mcsa04}
{\sc D. McDuff and D. Salamon},
\textit{$J$-holomorphic Curves and Symplectic Topology},
Amer. Math. Soc. Colloq. Publ. \textbf{52},
American Mathematical Society, Providence, RI (2004).
%
\bibitem{miln65}
{\sc J. Milnor},
\textit{Lectures on the $h$-Cobordism Theorem},
Princeton University Press, Princeton, NJ (1965).
%
\bibitem{nied06}
{\sc K. Niederkr\"uger},
The plastikstufe -- a generalization of the overtwisted disk to higher
dimensions,
\textit{Algebr. Geom. Topol.}
\textbf{6} (2006), 2473--2508.
%
\bibitem{suze}
{\sc S. Suhr and K. Zehmisch},
Linking and closed orbits,
preprint (2013),
{\tt arXiv:~1305.2799}.
%
\bibitem{wein91}
{\sc A. Weinstein},
Contact surgery and symplectic handlebodies,
\textit{Hokkaido Math. J.}
\textbf{20} (1991), 241--251.
%
\bibitem{zehm13}
{\sc K. Zehmisch},
The annulus property of simple holomorphic discs,
\textit{J. Symplectic Geom.}
{\bf 11} (2013), 135--161.
%
\end{thebibliography}
\end{document}